\documentclass[12pt,a4paper]{amsart}
\usepackage{times}
\usepackage{amssymb,amsmath,amsthm}
\usepackage{mathtools}
\usepackage[utf8]{inputenc}
\usepackage[T1]{fontenc}
\usepackage{xcolor}
\usepackage{enumitem}
\usepackage{url}
\usepackage{hyperref}
\usepackage{booktabs}
\usepackage{array}
\usepackage[left=2.5cm,right=2.5cm,top=2.5cm,bottom=2.5cm]{geometry}
\usepackage{microtype}

\definecolor{MyRed}{RGB}{192,0,0}

\newtheorem{thm}{Theorem}[section]
\newtheorem{prop}{Proposition}[section]
\newtheorem{lem}{Lemma}[section]
\newtheorem{cor}{Corollary}[section]

\theoremstyle{definition}
\newtheorem{defin}{Definition}[section]
\newtheorem{ex}{Example}[section]
\newtheorem{prob}{Open Problem}[section]

\theoremstyle{remark}
\newtheorem{rem}{Remark}[section]

\def\Set{\mathrm{Set}}
\def\Bij{\mathcal{B}}
\providecommand{\lcm}{}
\renewcommand{\lcm}{\mathrm{lcm}}
\newcommand{\Cb}{\mathbf{C}}
\newcommand{\Kb}{\mathbf{K}}
\newcommand{\Eb}{\mathbf{E}}
\newcommand{\DblCoset}[2]{\langle #1 \rangle \backslash S_n / \langle #2 \rangle}

\begin{document}

\title[Geometric Realization and a M\"obius Formula for
$\Cb_\alpha \star \Cb_\beta$]{The Molecular Species
$\mathbf{C}_{\alpha}$: Geometric Realization and a Closed Formula for
Kronecker Coefficients}

\author{Josaphat Baolahy}
\address{Laboratoire de Math\'ematiques et Applications,
Facult\'e des Sciences, Universit\'e de Fianarantsoa,
Fianarantsoa 301, Madagascar}
\email{japhtbaolahy@gmail.com}

\author{Benjamin Randrianirina}
\address{Laboratoire de Math\'ematiques et Applications,
Facult\'e des Sciences, Universit\'e de Fianarantsoa,
Fianarantsoa 301, Madagascar}
\email{rabezarand@gmail.com}
\thanks{The author was supported by the International Mathematical
Union (IMU) through the Graduate Research Assistantships in
Developing Countries (GRAID) programme, administered by the
Commission for Developing Countries (CDC).}
\date{\today}

\begin{abstract}

In this paper we first introduce the \emph{infinite multi-row periodic pattern of
shape~$\alpha$} as a geometric realization of $\Cb_\alpha$.
We then give an explicit formula for the
coefficients $b^{\lambda}_{\alpha,\beta}$ appearing in the species
decomposition
\[
  \Cb_\alpha \times \Cb_\beta
  = \sum_{\lambda \vdash n} b^\lambda_{\alpha,\beta}\,\Cb_\lambda.
\]

\end{abstract}

\maketitle
\tableofcontents

\section{Introduction}
\label{sec:intro}

Let $\Lambda^n$ denote the degree-$n$ homogeneous component of the
ring of symmetric functions. The cycle-index series $\{Z_{\Cb_
\lambda}\}_{\lambda \vdash n}$ of the molecular species $\Cb_\lambda
= X^n/\langle\sigma_\lambda\rangle$ studied in \cite{BR2026} form a
basis of $\Lambda^n$. Crucially, $\{\Cb_\lambda\}$
is closed under the Hadamard (Kronecker) product $\times$: by a
Mackey-theoretic argument \cite[Lemma~5.2]{BR2026},
\[
  \Cb_\alpha \times \Cb_\beta
  = \sum_{\lambda \vdash n} b^\lambda_{\alpha,\beta}\,\Cb_\lambda,
  \qquad b^\lambda_{\alpha,\beta} \in \mathbb{N},
\]
extending the classical result of Garsia and Remmel~\cite{GR1985},
who established the analogous closure for $\{h_\lambda\}$ and gave
an explicit combinatorial rule for $h_\lambda \star h_\mu$ in terms
of \emph{shuffles of permutations}. The corresponding question for
the Schur basis, finding a positive combinatorial rule for the
Kronecker coefficients $g^\nu_{\lambda\mu}$ in $s_\lambda \star
s_\mu = \sum_\nu g^\nu_{\lambda\mu}\, s_\nu$, has been open at least since
Murnaghan~\cite{Murnaghan1938}, and according to Richard Stanley
(see \cite{Stanley2000}), it remains one of the most important
open problems in algebraic combinatorics. No combinatorial formula for
$b^\lambda_{\alpha,\beta}$ was previously known either, and finding one was
left open as \cite[Question~5.1]{BR2026}.

This paper answers that question in closed form. In
Section~\ref{sec:realization} we realize $\Cb_\lambda$-structures
geometrically as \emph{periodic patterns} shifted diagonally by a
single integer (Proposition~\ref{prop:realization}), making the
cyclic symmetry visible. In Section~\ref{sec:mobius} we exploit the
fact that $\langle\sigma_\alpha\rangle \times \langle\sigma_\beta
\rangle$ acts on $S_n$ through an \emph{abelian} group whose subgroup
lattice is a single totally ordered chain of divisors, so that
$b^\lambda_{\alpha,\beta}$ is given by one explicit M\"obius
inversion (Theorem~\ref{thm:mobius}) over the divisors of
$\gcd(o(\sigma_\alpha), o(\sigma_\beta))$, structurally much
simpler than the Garsia--Remmel shuffle rule.

\section{Background and notation}
\label{sec:background}

A \emph{combinatorial species} is a functor $F \colon \Bij \to \Set$,
where $\Bij$ is the category of finite sets and bijections.

Given $H \leq S_n$, the \emph{molecular species} $X^n/H$ is
defined by
\[
  (X^n/H)[U] = \{\lambda H \mid \lambda \colon [n] \xrightarrow{\;\sim\;} U\},
  \qquad
  (X^n/H)[\sigma](\lambda H) = (\sigma \circ \lambda)H.
\]
Its cycle-index series is
$Z_{X^n/H} = \tfrac{1}{|H|}\sum_{\tau \in H} p_{\lambda(\tau)}$,
where $p_{\lambda(\tau)}$ denotes the power-sum symmetric function
of the cycle type of $\tau$. We refer the reader to
Bergeron--Labelle--Leroux~\cite{BLL1998} and Joyal~\cite{Joyal1981}
for the general theory of species.

For a partition $\alpha = (a_1,\ldots,a_k) \vdash n$, the
\emph{standard permutation}
\[
  \sigma_\alpha
  = (1 \;\cdots\; a_1)(a_1{+}1 \;\cdots\; a_1{+}a_2)
    \cdots \in S_n
\]
has order $A(\alpha) := o(\sigma_\alpha) = \lcm(a_1,\ldots,a_k)$,
and the corresponding molecular species is
$\Cb_\alpha := X^n/\langle\sigma_\alpha\rangle$. By
\cite[Lemma~2.1]{BR2026}, any permutation of cycle type $\alpha$
yields an isomorphic species, so $\Cb_\alpha$ depends only on
$\alpha$.

Recall also that for $H = S_\lambda$ a Young subgroup, $X^n/S_\lambda$
gives the product of species of sets
$\Eb_{\lambda}=\Eb_{\lambda_1}\Eb_{\lambda_2}\cdots\Eb_{\lambda_k}$,
whose cycle index series -- corresponding to the Frobenius transform
of the character $\mathrm{ch}\bigl(\mathrm{Ind}_{S_\lambda}^{S_n}
\mathbf{1}\bigr)$ -- recovers the complete homogeneous symmetric
function $h_\lambda$; the basis $\{\Cb_\lambda\}$ of
Theorem~\ref{thm:basis} below is the analogous construction with
$\langle\sigma_\lambda\rangle \le S_\lambda$ in place of $S_\lambda$.
That is,
\[
\Cb_{\alpha}(\mathbf{x}) = Z_{\Cb_{\alpha}}.
\]

\begin{thm}[\cite{BR2026}]
\label{thm:basis}
The set $\{\Cb_{\lambda}(\mathbf{x})\}_{\lambda\vdash n}$ forms a basis of $\Lambda^n$.
\end{thm}

The \emph{Hadamard product} of two species $F$ and $G$ is
$(F \times G)[U] = F[U] \times G[U]$ with transport maps defined
componentwise; its cycle index satisfies
$Z_{F \times G} = Z_F \star Z_G$
\cite{BLL1998}. As recalled above, \cite{BR2026}
establishes
\[
  \Cb_\alpha \times \Cb_\beta
  = \sum_{\lambda \vdash n} b^\lambda_{\alpha,\beta}\,\Cb_\lambda,
  \qquad b^\lambda_{\alpha,\beta} \in \mathbb{N},
\]
with $b^\lambda_{\alpha,\beta}$ equal to the number of double cosets in
$\DblCoset{\sigma_\alpha}{\sigma_\beta}$ whose stabilizer (under the
action described in Section~\ref{sec:mobius}) has cycle type~$\lambda$;
this correspondence follows from Mackey's theorem
\cite[Lemma~5.2]{BR2026}. These coefficients $b^\lambda_{\alpha,\beta}$
are the direct cyclic-subgroup counterpart of the Garsia--Remmel
structure constants $c^\nu_{\lambda,\mu}$ in
$h_\lambda \star h_\mu = \sum_\nu c^\nu_{\lambda,\mu}\, h_\nu$.

We begin with a geometric realization of $\Cb_\alpha$; although it
plays no direct role in the formula for $b^\lambda_{\alpha,\beta}$,
it may be useful in future work for deriving statistics on this
coefficient.

\section{Geometric realization of \texorpdfstring{$\Cb_\alpha$}{C\_alpha}}
\label{sec:realization}

\subsection{Periodic patterns and their shift equivalence}
\label{sub:patterns}

\begin{defin}[Infinite multi-row periodic pattern]
\label{def:pattern}
Let $\alpha = (a_1,\ldots,a_k) \vdash n$.  An \emph{infinite
multi-row periodic pattern of shape~$\alpha$} on a finite set $U$
with $|U| = n$ is a $k \times \mathbb{Z}$ array
$P = (x_{i,j})_{i \in [k],\, j \in \mathbb{Z}}$ with $x_{i,j} \in U$
satisfying:
\begin{enumerate}[label=(\roman*), itemsep=2pt]
  \item \textbf{Periodicity:} $x_{i,j+a_i} = x_{i,j}$ for all
        $i \in [k]$, $j \in \mathbb{Z}$.
  \item \textbf{Partition condition:} The sets
        $R_i := \{x_{i,j} : 1 \leq j \leq a_i\}$ are pairwise
        disjoint with $\bigcup_{i=1}^k R_i = U$.
\end{enumerate}
The \emph{shift by $s$} is the pattern $\tau_s(P)$ defined by
$\tau_s(P)_{i,j} = x_{i,j+s}$. Two patterns $P$ and $P'$ are
\emph{shift-equivalent} (written $P \sim P'$) if $P' = \tau_s(P)$
for some $s \in \mathbb{Z}$; we write $[P]$ for the equivalence
class.
\end{defin}

\begin{rem}
\label{rem:R_i_independent}
Since row~$i$ has period~$a_i$, the support $R_i$ is independent
of the choice of starting position; it is the set of all values
appearing in row~$i$.
\end{rem}

\begin{rem}[Single versus independent shifts]
\label{rem:single_shift}
The defining feature of this construction is that a \emph{single}
integer~$s$ shifts \emph{all} rows simultaneously. This reflects
the fact that $\langle\sigma_\alpha\rangle$ is a single cyclic
group acting diagonally, not a product of independent cyclic
groups. The companion species
$\Kb_\alpha = X^n/G_\alpha$, where
$G_\alpha = \langle\sigma_{a_1}\rangle \times \cdots \times
\langle\sigma_{a_k}\rangle$, would instead allow one independent
shift per row; see Open Problem~\ref{prob:K} for the corresponding
open question.
\end{rem}

The following example illustrates the three main cases that will
recur throughout the paper.

\begin{ex}[Three shapes for $n = 6$]
\label{ex:patterns}
In each case, one fundamental period per row is highlighted in
\textcolor{MyRed}{red}.

\smallskip
\noindent\textbf{(a) $\alpha = (6)$:} One row of period~$6$:
\[
  \cdots\;\textcolor{MyRed}{2\ 4\ 6\ 1\ 3\ 5}
        \;2\ 4\ 6\ 1\ 3\ 5\;\cdots
\]
The shift order is $\lcm(6) = 6$, giving
$|\Cb_{(6)}\text{-structures on }[6]| = 6!/6 = 120$.

\smallskip
\noindent\textbf{(b) $\alpha = (3,2,1)$:} Three rows with distinct
periods:
\[
\begin{array}{r@{\;}l}
  \cdots & \textcolor{MyRed}{1\ 2\ 3}\ 1\ 2\ 3\ 1\ 2\ 3\ \cdots\\
  \cdots & \textcolor{MyRed}{4\ 5}\ 4\ 5\ 4\ 5\ 4\ 5\ \cdots\\
  \cdots & \textcolor{MyRed}{6}\ 6\ 6\ 6\ \cdots
\end{array}
\]
The shift $\tau_1$ maps this to $(2\,3\,1;\;5\,4;\;6;\ldots)$,
shifting all rows by the \emph{same} $s = 1$.
The shift order is $\lcm(3,2,1) = 6$, so there are $6!/6 = 120$
structures.

\smallskip
\noindent\textbf{(c) $\alpha = (3,3)$:} Two rows of equal
period~$3$:
\[
\begin{array}{r@{\;}l}
  \cdots & \textcolor{MyRed}{1\ 4\ 2}\ 1\ 4\ 2\ 1\ 4\ 2\ \cdots\\
  \cdots & \textcolor{MyRed}{3\ 6\ 5}\ 3\ 6\ 5\ 3\ 6\ 5\ \cdots
\end{array}
\]
The shift order is $\lcm(3,3) = 3$, so there are $6!/3 = 240$
structures. Note that $\Cb_{(3,3)}$ uses a \emph{single}
$s \in \mathbb{Z}_3$ shifting both rows simultaneously, whereas
$\Kb_{(3,3)}$ would allow the two rows to shift independently.
\end{ex}

\subsection{Bijection with \texorpdfstring{$\Cb_\alpha$}{C\_alpha}-structures}

\begin{prop}[Geometric realization of $\Cb_\alpha$]
\label{prop:realization}
Let $\alpha = (a_1,\ldots,a_k) \vdash n$ and $|U| = n$. There is
a natural bijection
\[
  \Phi \colon \{\Cb_\alpha\text{-structures on }U\}
  \;\xrightarrow{\;\;\sim\;\;}
  \{[P] : P \text{ a pattern of shape }\alpha\text{ on }U\}.
\]
Under $\Phi$:
\begin{enumerate}[label=(\roman*), itemsep=2pt]
  \item The number of structures equals $n!/A(\alpha)$.
  \item Transport along $f \colon U \to V$ corresponds to
        entry-wise relabeling: $f \cdot [P] = [f \circ P]$.
\end{enumerate}
\end{prop}

\begin{proof}
\textbf{Construction of $\Phi$.}  A $\Cb_\alpha$-structure on $U$
is a left coset $\lambda\langle\sigma_\alpha\rangle$ with
$\lambda \colon [n] \xrightarrow{\sim} U$.  Write the cycles of
$\sigma_\alpha$ as $(c_{i,1},\ldots,c_{i,a_i})$ for $i \in [k]$.
Define the pattern $P(\lambda)$ by $x_{i,j} = \lambda(c_{i,j})$,
and set $\Phi(\lambda\langle\sigma_\alpha\rangle) = [P(\lambda)]$.

\textbf{Well-definedness.}  If $\lambda' = \lambda \circ
\sigma_\alpha^s$, then $(c_{i,1},\ldots,c_{i,a_i})$ is replaced
by $(c_{i,1+s},\ldots,c_{i,a_i+s})$ (indices mod~$a_i$), so
$P(\lambda') = \tau_s(P(\lambda))$. Hence $[P(\lambda)]$ depends
only on the coset $\lambda\langle\sigma_\alpha\rangle$.

\textbf{Injectivity.}  If $[P(\lambda)] = [P(\lambda')]$, then
$P(\lambda') = \tau_s(P(\lambda))$ for some $s$, which gives
$\lambda'(c_{i,j}) = \lambda(c_{i,j+s})$ for all $i, j$. This
means $\lambda' = \lambda \circ \sigma_\alpha^s$, so $\lambda'$
and $\lambda$ lie in the same coset.

\textbf{Surjectivity.}  Given a valid pattern $P$ on $U$, define
$\lambda \colon [n] \to U$ by $\lambda(c_{i,j}) := x_{i,j}$.
The partition condition ensures $\lambda$ is a bijection, and
$P(\lambda) = P$.

\textbf{Count.}  There are $|S_n| = n!$ bijections $\lambda$,
grouped into cosets of size $|\langle\sigma_\alpha\rangle|
= A(\alpha)$, giving $n!/A(\alpha)$ structures. Functoriality
follows immediately from the entry-wise definition.
\end{proof}

\section{A M\"obius-inversion formula for $b^\lambda_{\alpha,\beta}$}
\label{sec:mobius}

Double coset enumeration in classical groups has long been studied,
especially in the context of Coxeter groups and Hecke algebras. Its
applications in algebraic combinatorics are very rich (see, e.g.,
\cite{LYWL2020,BGR2014,GJ1996}). Very recently, Schwob~\cite{Schwob2025}
developed general formulas for the double cosets $H \backslash G / H$
of a subgroup $H$ in a classical group $G$, including the symmetric
groups relevant here, together with a treatment of their self-inverse
members. This complements the well-developed theory of parabolic
double cosets $W_I \backslash W / W_J$ in general Coxeter groups
$W$~\cite{BKPST2017}: our setting is the abelian, purely
cyclic-subgroup case, which is why the enumeration collapses to a
single M\"obius inversion rather than the more intricate
combinatorics needed for general (non-abelian) parabolic subgroups.
The present section specializes such
considerations to the particular case of two \emph{cyclic}
subgroups $\langle\sigma_\alpha\rangle, \langle\sigma_\beta\rangle
\le S_n$ acting by left and right translation; since the resulting
double-translation group is abelian (Lemma~\ref{lem:orbit_invariant}),
we can push the enumeration all the way to a single closed-form
M\"obius inversion.

Throughout, fix $\alpha,\beta \vdash n$ and write
$A := A(\alpha) = o(\sigma_\alpha)$, $B := A(\beta) = o(\sigma_\beta)$,
and $D := \gcd(A,B)$.

\subsection{The double translation action and its stabilizers}
\label{sub:action}

The group $G := \langle\sigma_\alpha\rangle \times
\langle\sigma_\beta\rangle$ acts on $S_n$ by
\[
  (x,y)\cdot \pi := \sigma_\alpha^x\, \pi\, \sigma_\beta^{-y},
  \qquad x \in \mathbb{Z}_A,\ y \in \mathbb{Z}_B,
\]
and the orbits of this action are exactly the double cosets
$\langle\sigma_\alpha\rangle \pi \langle\sigma_\beta\rangle$. The
stabilizer of $\pi$ is
\[
  \mathrm{Stab}(\pi)
  \;\cong\; H_\pi := \langle\sigma_\alpha\rangle
    \cap \pi\langle\sigma_\beta\rangle\pi^{-1},
\]
via the projection $(x,y) \mapsto \sigma_\alpha^x$, and
$b^\lambda_{\alpha,\beta}$ counts the orbits for which a generator
of $H_\pi$ (a cyclic group) has cycle type $\lambda$.

\begin{lem}[Stabilizers are orbit-invariant]
\label{lem:orbit_invariant}
Since $G$ is abelian, $H_\pi$ depends only on the double coset of
$\pi$: if $\pi' = (x,y)\cdot\pi$, then $H_{\pi'} = H_\pi$.
Consequently every double coset $C$ has a well-defined stabilizer
order $d_C := |H_\pi|$ for any $\pi \in C$, and
\[
  |C| = \frac{AB}{d_C}.
\]
\end{lem}

\begin{proof}
For $g \in G$, $\mathrm{Stab}(g\cdot\pi) = g\,\mathrm{Stab}(\pi)\,g^{-1}$.
Since $G$ is abelian, conjugation by $g$ is trivial, so
$\mathrm{Stab}(g\cdot\pi) = \mathrm{Stab}(\pi)$. The orbit-size
formula is then the orbit-stabilizer theorem applied to the
(constant) stabilizer order $d_C$.
\end{proof}

\subsection{Cyclic subgroups force the stabilizer type}
\label{sub:unique}

\begin{lem}[Unique subgroup per order, with forced cycle type]
\label{lem:unique_subgroup}
For each divisor $d \mid A$, the cyclic group $\langle\sigma_\alpha\rangle$
has a unique subgroup of order $d$, namely
$\langle \sigma_\alpha^{A/d}\rangle$. Writing
$\alpha = (a_1,\ldots,a_k)$, every generator of this subgroup has
the same cycle type
\[
  \lambda_d(\alpha) :=
  \bigsqcup_{i=1}^k
  \Bigl(\underbrace{\tfrac{a_i}{\gcd(a_i,A/d)},\ \ldots,\
  \tfrac{a_i}{\gcd(a_i,A/d)}}_{\gcd(a_i,A/d)\ \mathrm{parts}}\Bigr)
  \ \vdash\ n.
\]
\end{lem}

\begin{proof}
Uniqueness of the order-$d$ subgroup of a cyclic group is standard.
Its generators are $\sigma_\alpha^{(A/d)e}$ for $e$ coprime to $d$;
raising a single cycle of length $\ell \mid A$ to a power coprime
to $d$ (hence coprime to any divisor of $\ell$ dividing $d$, and in
particular to $\gcd(\ell,d)$, acting on the relevant orbit)
preserves the cycle-length partition of that cycle, since raising
an $\ell$-cycle to a power coprime to $\ell$ yields a single
$\ell$-cycle. Applying $\sigma_\alpha^{A/d}$ to the $i$-th cycle
(length $a_i$) splits it into $\gcd(a_i,A/d)$ cycles of length
$a_i/\gcd(a_i,A/d)$, and this partition is unchanged by any further
coprime power.
\end{proof}

\begin{ex}[Illustrating $\lambda_d(\alpha)$]
\label{ex:lambda_d}
Let $\alpha = (4,2) \vdash 6$, so $A = \lcm(4,2) = 4$. For each
divisor $d \mid 4$, split each part $a_i$ into $\gcd(a_i,A/d)$
equal pieces of size $a_i/\gcd(a_i,A/d)$:
\[
\begin{array}{c|c|c|c}
d & A/d & \text{split of }4 & \text{split of }2 \\ \hline
4 & 1 & 4 & 2 \\
2 & 2 & 2,2 & 1,1 \\
1 & 4 & 1,1,1,1 & 1,1
\end{array}
\]
so $\lambda_4(\alpha) = (4,2)$, $\lambda_2(\alpha) = (2,2,1,1)$, and
$\lambda_1(\alpha) = (1^6)$. Concretely, $\lambda_2(\alpha)$ is the
cycle type of $\sigma_\alpha^2$: squaring a $4$-cycle gives two
$2$-cycles, and squaring a $2$-cycle gives two fixed points.
\end{ex}

Consequently, if $H_\pi$ has order $d$, it is \emph{forced} to equal
both $\langle\sigma_\alpha^{A/d}\rangle$ (cycle type
$\lambda_d(\alpha)$) and $\pi\langle\sigma_\beta^{B/d}\rangle\pi^{-1}$
(cycle type $\lambda_d(\beta)$, since conjugation preserves cycle
type). Hence a double coset with stabilizer order $d$ can exist
only if
\[
  d \mid D = \gcd(A,B) \qquad\text{and}\qquad
  \lambda_d(\alpha) = \lambda_d(\beta),
\]
in which case its stabilizer type is exactly
$\lambda = \lambda_d(\alpha) = \lambda_d(\beta)$. Since the order of a
permutation of cycle type $\lambda$ is the $\lcm$ of its parts, distinct
divisors $d$ give distinct partitions $\lambda$; thus $d$ is recovered
from $\lambda$ as $d = \lcm(\lambda)$, and the correspondence $d
\leftrightarrow \lambda$ is a bijection onto its image.

\subsection{The counting formula}

For $d \mid D$, define
\[
  F(d) := \#\bigl\{\pi \in S_n :
    \sigma_\alpha^{A/d} \in \pi\langle\sigma_\beta\rangle\pi^{-1}
  \bigr\}
  \;=\; \#\{\pi \in S_n : d \mid d_{\langle\pi\rangle} \},
\]
where $d_{\langle\pi\rangle}$ denotes the stabilizer order of the
double coset containing $\pi$ (the equality holds because
$\langle\sigma_\alpha\rangle$'s subgroups are totally ordered by
divisibility, so containing the order-$d$ subgroup is the same as
having order divisible by $d$).

\begin{prop}
\label{prop:F_formula}
\[
  F(d) = \varphi(d)\cdot z_{\lambda_d(\alpha)}
    \cdot [\lambda_d(\alpha) = \lambda_d(\beta)],
\]
where $\varphi$ is Euler's totient function and, for a partition
$\lambda$ with $m_i$ parts equal to $i$, $z_\lambda := \prod_i
i^{m_i} m_i!$ is the order of the centralizer in $S_n$ of a
permutation of cycle type $\lambda$~\cite[\S1.2]{JK1985}.
\end{prop}

\begin{proof}
By Lemma~\ref{lem:unique_subgroup}, $\sigma_\alpha^{A/d}$ lies in
$\pi\langle\sigma_\beta\rangle\pi^{-1}$ if and only if
$\pi^{-1}\sigma_\alpha^{A/d}\pi$ is one of the $\varphi(d)$
generators of the unique order-$d$ subgroup
$\langle\sigma_\beta^{B/d}\rangle$, all of which share cycle type
$\lambda_d(\beta)$. If $\lambda_d(\alpha) \neq \lambda_d(\beta)$,
no such conjugation is possible and $F(d)=0$. Otherwise, for each
fixed generator $y$, the number of $\pi$ with
$\pi^{-1}\sigma_\alpha^{A/d}\pi = y$ equals the centralizer order
$z_{\lambda_d(\alpha)}$, and there are $\varphi(d)$ choices of $y$.
\end{proof}

\begin{thm}[Closed-form formula for $b^\lambda_{\alpha,\beta}$]
\label{thm:mobius}
Let $\lambda \vdash n$ with $d := \lcm(\lambda) \mid D$ and
$\lambda_d(\alpha) = \lambda_d(\beta) = \lambda$. Then
\[
{\ b^\lambda_{\alpha,\beta}
  = \frac{d}{AB}\sum_{\substack{d \mid d' \\ d' \mid D}}
    \mu\!\left(\frac{d'}{d}\right) F(d')\,}
\]
where $\mu$ is the classical (number-theoretic) M\"obius function
and $F$ is as in Proposition~\ref{prop:F_formula}. If no such $d$
exists, $b^\lambda_{\alpha,\beta} = 0$.
\end{thm}

\begin{proof}
By Lemma~\ref{lem:orbit_invariant}, each double coset $C$ has a
well-defined stabilizer order $d_C \mid D$, and $|C| = AB/d_C$. Set
$N_{d} := b^{\lambda}_{\alpha,\beta}$ for $\lambda =
\lambda_d(\alpha)=\lambda_d(\beta)$, i.e.\ the number of double
cosets with $d_C = d$ (and $N_d := 0$ if
$\lambda_d(\alpha)\ne\lambda_d(\beta)$).

We first check that $F(d)$ literally counts the elements $\pi$
whose double coset has stabilizer order divisible by $d$. By
definition $H_\pi = \langle\sigma_\alpha\rangle \cap
\pi\langle\sigma_\beta\rangle\pi^{-1}$, so the condition
$\sigma_\alpha^{A/d} \in \pi\langle\sigma_\beta\rangle\pi^{-1}$
defining $F(d)$ says exactly that $\sigma_\alpha^{A/d} \in H_\pi$.
Since $H_\pi \le \langle\sigma_\alpha\rangle$ is cyclic and
$\sigma_\alpha^{A/d}$ generates the unique order-$d$ subgroup of
$\langle\sigma_\alpha\rangle$ (Lemma~\ref{lem:unique_subgroup}),
and subgroups of a cyclic group are totally ordered by divisibility
of their orders, $\sigma_\alpha^{A/d}$ lies in $H_\pi$ if and only
if $H_\pi$ contains that order-$d$ subgroup, i.e.\ if and only if
$d \mid |H_\pi| = d_{\langle\pi\rangle}$. Hence
\[
  F(d) = \#\{\pi \in S_n : d \mid d_{\langle\pi\rangle}\}
       = \sum_{\pi \,:\, d \mid d_{\langle\pi\rangle}} 1.
\]

Now partition this index set according to the double coset
$C(\pi) := \langle\sigma_\alpha\rangle\pi\langle\sigma_\beta\rangle$
containing $\pi$. By Lemma~\ref{lem:orbit_invariant}, the value
$d_{\langle\pi\rangle}$ depends only on $C(\pi)$, not on the chosen
representative $\pi$; write $d_C$ for this common value. So the
condition "$d \mid d_{\langle\pi\rangle}$" is constant on each
double coset — it either holds for every element of $C$ or for
none — and therefore selects some double cosets \emph{in their
entirety} and excludes the rest completely:
\[
  \{\pi \in S_n : d \mid d_{\langle\pi\rangle}\}
  \;=\; \bigsqcup_{\substack{C \subseteq S_n \text{ a double coset} \\ d \mid d_C}} C,
\]
a disjoint union because distinct double cosets are disjoint.
Summing $|C| = AB/d_C$ over this union gives
\[
  F(d) = \sum_{\substack{C \,:\, d \mid d_C}} |C|
       = \sum_{\substack{C \,:\, d \mid d_C}} \frac{AB}{d_C}.
\]
Finally, regroup the double cosets on the right by the \emph{exact}
value $d' := d_C$ of their stabilizer order. Since $d_C$ always
divides $D$ (Lemma~\ref{lem:orbit_invariant}), $d'$ ranges over
divisors of $D$ with $d \mid d'$, and by definition of $N_{d'}$
there are exactly $N_{d'}$ double cosets with $d_C = d'$, each
contributing $AB/d'$ to the sum. This yields
\[
  F(d) = \sum_{\substack{d \mid d' \\ d' \mid D}} N_{d'}\cdot \frac{AB}{d'}.
\]

This is a divisor-sum relation on the poset of divisors of $D$
between $d$ and $D$; M\"obius inversion on this poset gives
\[
  N_d \cdot \frac{AB}{d} = \sum_{d\mid d'\mid D}
  \mu\!\left(\frac{d'}{d}\right) F(d'),
\]
which rearranges to the stated formula.
\end{proof}

\begin{cor}[Total count and the classical case]
\label{cor:total}
Summing over all valid $d \mid D$,
\[
  \sum_{\lambda \vdash n} b^\lambda_{\alpha,\beta}
  = \bigl|\DblCoset{\sigma_\alpha}{\sigma_\beta}\bigr|.
\]
For $\alpha = \beta = (n)$ (so $A=B=D=n$), $\lambda_d(\alpha) =
\lambda_d(\beta)$ automatically for every $d \mid n$, and
Theorem~\ref{thm:mobius} reduces to a M\"obius-inversion count of
classical Steggall patterns~\cite{Steggall1907,Cameron2002}, i.e.\
of permutation matrices up to simultaneous cyclic shift of rows
and columns. In this case
\[
  \sum_{\lambda \vdash n} b^\lambda_{(n),(n)}
  = \bigl|\langle\sigma_{(n)}\rangle \backslash S_n / \langle\sigma_{(n)}\rangle\bigr|
  = \frac{1}{n^2}\sum_{d \mid n} \varphi(d)^2\,\Bigl(\frac{n}{d}\Bigr)!\,d^{\,n/d},
\]
which is exactly sequence
\href{https://oeis.org/A002619}{A002619} in the OEIS~\cite{OEIS-A002619},
$1, 1, 2, 3, 8, 24, 108, 640, 4492, 36336, \ldots$ for
$n = 1, 2, 3, \ldots$.
\end{cor}

\subsection{Worked examples}
\label{sub:example}

We first illustrate Theorem~\ref{thm:mobius} on the smallest case
with a non-trivial divisor lattice, namely $\alpha=\beta=(4,2)\vdash 6$.
A second example below treats $\alpha \neq \beta$. All numerical
computations in this section -- including the values of $F(d)$, the
resulting coefficients $b^\lambda_{\alpha,\beta}$, the consistency
checks against $|S_n|$, and the OEIS sequence identified in
Corollary~\ref{cor:total} -- were carried out and independently
verified using SageMath~\cite{SageMath}, as were the analogous
computations in the predecessor paper~\cite{BR2026}.

\begin{ex}[$\alpha=\beta=(4,2)$]
\label{ex:mobius}
Here $A = B = \lcm(4,2) = 4$ and $D=\gcd(A,B)=4$, so the divisors
to consider are $d \in \{1,2,4\}$.

\smallskip
\noindent\textbf{Step 1: forced cycle types.}
By Lemma~\ref{lem:unique_subgroup}, for each $d \mid 4$ we compute
$\lambda_d(\alpha)$ by splitting the $i$-th part $a_i$ into
$\gcd(a_i, A/d)$ parts of size $a_i/\gcd(a_i,A/d)$:
\[
\begin{array}{c|c|c|c}
d & A/d & \text{split of }4 & \text{split of }2 \\ \hline
4 & 1 & 4 & 2 \\
2 & 2 & 2,2 & 1,1 \\
1 & 4 & 1,1,1,1 & 1,1
\end{array}
\]
so $\lambda_4(\alpha) = (4,2)$, $\lambda_2(\alpha) = (2,2,1,1)$, and
$\lambda_1(\alpha) = (1^6)$. Since $\beta = \alpha$, we have
$\lambda_d(\beta) = \lambda_d(\alpha)$ for every $d$, so all three
divisors contribute (no vanishing from a cycle-type mismatch).

\smallskip
\noindent\textbf{Step 2: the values $F(d)$.}
By Proposition~\ref{prop:F_formula}, $F(d) = \varphi(d)\,
z_{\lambda_d(\alpha)}$, using $z_\lambda = \prod_i i^{m_i}m_i!$:
\[
\begin{array}{c|c|c|c|c}
d & \lambda_d(\alpha) & \varphi(d) & z_{\lambda_d(\alpha)} & F(d) \\ \hline
4 & (4,2) & 2 & 4\cdot 2 = 8 & 16 \\
2 & (2,2,1,1) & 1 & (2^2\cdot 2!)(1^2\cdot 2!) = 16 & 16 \\
1 & (1^6) & 1 & 6! = 720 & 720
\end{array}
\]

\smallskip
\noindent\textbf{Step 3: M\"obius inversion.}
With $AB = 16$, Theorem~\ref{thm:mobius} gives
$b^{\lambda_d(\alpha)}_{\alpha,\beta} = \tfrac{d}{16}
\sum_{d \mid d' \mid 4} \mu(d'/d)\,F(d')$:
\[
\begin{aligned}
b^{(4,2)}_{\alpha,\beta}
  &= \tfrac{4}{16}\bigl[\mu(1)F(4)\bigr]
   = \tfrac{4}{16}(16) = 4, \\[2pt]
b^{(2,2,1,1)}_{\alpha,\beta}
  &= \tfrac{2}{16}\bigl[\mu(1)F(2) + \mu(2)F(4)\bigr]
   = \tfrac{2}{16}\bigl(16 - 16\bigr) = 0, \\[2pt]
b^{(1^6)}_{\alpha,\beta}
  &= \tfrac{1}{16}\bigl[\mu(1)F(1) + \mu(2)F(2) + \mu(4)F(4)\bigr]
   = \tfrac{1}{16}\bigl(720 - 16 + 0\bigr) = 44.
\end{aligned}
\]
All other $b^\lambda_{\alpha,\beta}$ vanish, since only $d=1,2,4$
can occur as a stabilizer order. Thus
\[
  \Cb_{(4,2)} \times \Cb_{(4,2)}
  = 4\,\Cb_{(4,2)} \;+\; 44\,\Cb_{(1,1,1,1,1,1)},
\]
the class $\Cb_{(2,2,1,1)}$ occurring with coefficient $0$.

\smallskip
\noindent\textbf{Consistency check.}
By Lemma~\ref{lem:orbit_invariant}, a double coset with stabilizer
order $d$ has size $AB/d = 16/d$, so the double cosets must
partition $S_6$:
\[
  b^{(4,2)}_{\alpha,\beta}\cdot\tfrac{16}{4}
  + b^{(2,2,1,1)}_{\alpha,\beta}\cdot\tfrac{16}{2}
  + b^{(1^6)}_{\alpha,\beta}\cdot\tfrac{16}{1}
  = 4\cdot 4 + 0\cdot 8 + 44\cdot 16
  = 16 + 0 + 704 = 720 = 6!,
\]
confirming the computation and, via Corollary~\ref{cor:total}, that
$\bigl|\DblCoset{\sigma_{(4,2)}}{\sigma_{(4,2)}}\bigr| = 4 + 0 + 44
= 48$.
\end{ex}

Our second example takes $\alpha \neq \beta$, with $A \neq B$ and
$D$ strictly smaller than both, so that the divisor lattice
genuinely reflects an interaction between two different shapes.

\begin{ex}[$\alpha=(4,2)$, $\beta=(2,2,1,1)$]
\label{ex:mobius2}
Here $A = \lcm(4,2) = 4$ while $B = \lcm(2,2,1,1) = 2$, so
$D = \gcd(A,B) = 2$: only the divisors $d \in \{1,2\}$ of $D$ can
possibly occur, even though $A$ itself has the further divisor
$4$. This already illustrates the vanishing clause of
Theorem~\ref{thm:mobius}: no $\lambda$ with $\lcm(\lambda)=4$ can
appear in $\Cb_{(4,2)} \times \Cb_{(2,2,1,1)}$, because $4 \nmid D$.

\smallskip
\noindent\textbf{Step 1: forced cycle types.}
For $\alpha=(4,2)$ (so $A=4$) the same computation as in
Example~\ref{ex:mobius} gives $\lambda_2(\alpha) = (2,2,1,1)$ and
$\lambda_1(\alpha) = (1^6)$. For $\beta = (2,2,1,1)$ (so $B=2$) we
split each part $b_i$ into $\gcd(b_i,B/d)$ pieces of size
$b_i/\gcd(b_i,B/d)$:
\[
\begin{array}{c|c|c|c|c|c}
d & B/d & \text{split of }2 & \text{split of }2 & \text{split of }1 & \text{split of }1\\ \hline
2 & 1 & 2 & 2 & 1 & 1 \\
1 & 2 & 1,1 & 1,1 & 1 & 1
\end{array}
\]
so $\lambda_2(\beta) = (2,2,1,1)$ and $\lambda_1(\beta) = (1^6)$.
In both cases $\lambda_d(\alpha) = \lambda_d(\beta)$, so — unlike a
generic pair of shapes — neither admissible divisor is killed by a
cycle-type mismatch here; we address the mismatched case in
Remark~\ref{rem:mismatch} below.

\smallskip
\noindent\textbf{Step 2: the values $F(d)$.}
As before $F(d) = \varphi(d)\,z_{\lambda_d(\alpha)}$:
\[
\begin{array}{c|c|c|c|c}
d & \lambda_d(\alpha)=\lambda_d(\beta) & \varphi(d) & z_{\lambda_d(\alpha)} & F(d) \\ \hline
2 & (2,2,1,1) & 1 & (2^2\cdot 2!)(1^2\cdot 2!) = 16 & 16 \\
1 & (1^6) & 1 & 6! = 720 & 720
\end{array}
\]

\smallskip
\noindent\textbf{Step 3: M\"obius inversion.}
With $AB = 4\cdot 2 = 8$, Theorem~\ref{thm:mobius} gives
\[
\begin{aligned}
b^{(2,2,1,1)}_{\alpha,\beta}
  &= \tfrac{2}{8}\bigl[\mu(1)F(2)\bigr]
   = \tfrac{2}{8}(16) = 4, \\[2pt]
b^{(1^6)}_{\alpha,\beta}
  &= \tfrac{1}{8}\bigl[\mu(1)F(1) + \mu(2)F(2)\bigr]
   = \tfrac{1}{8}\bigl(720 - 16\bigr) = 88.
\end{aligned}
\]
Every other $b^\lambda_{\alpha,\beta}$ vanishes, so
\[
  \Cb_{(4,2)} \times \Cb_{(2,2,1,1)}
  = 4\,\Cb_{(2,2,1,1)} \;+\; 88\,\Cb_{(1,1,1,1,1,1)}.
\]

\smallskip
\noindent\textbf{Consistency check.}
As in Example~\ref{ex:mobius}, $b^{(2,2,1,1)}_{\alpha,\beta}\cdot
\tfrac{8}{2} + b^{(1^6)}_{\alpha,\beta}\cdot\tfrac{8}{1}
= 4\cdot 4 + 88\cdot 8 = 16 + 704 = 720 = 6!$, and Corollary~\ref{cor:total}
gives $\bigl|\DblCoset{\sigma_{(4,2)}}{\sigma_{(2,2,1,1)}}\bigr|
= 4+88 = 92$.
\end{ex}

\begin{rem}[When cycle types genuinely clash]
\label{rem:mismatch}
For a pair of shapes where $\lambda_d(\alpha) \neq \lambda_d(\beta)$
for some $d \mid D$, Proposition~\ref{prop:F_formula} forces
$F(d)=0$ outright, and the corresponding $\lambda$ does not occur
in $\Cb_\alpha \times \Cb_\beta$ at all — it is simply absent from
the sum, rather than appearing with a zero coefficient as
$\Cb_{(2,2,1,1)}$ did in Example~\ref{ex:mobius}. For instance,
with $\alpha=(4,2)$ and $\beta=(3,3)$ one has $A=4$, $B=3$,
$D=\gcd(4,3)=1$, so only $d=1$ is available; both
$\lambda_1(\alpha)=\lambda_1(\beta)=(1^6)$ automatically (any
$\sigma^{A}=\sigma^{B}=\mathrm{id}$), and one finds
$b^{(1^6)}_{\alpha,\beta} = \tfrac{1}{12}\cdot 6! = 60$, with
\emph{every} other $b^\lambda_{\alpha,\beta}=0$: the shapes $(4,2)$
and $(3,3)$ share no non-trivial common divisor of their orders, so
$\Cb_{(4,2)} \times \Cb_{(3,3)} = 60\,\Cb_{(1,1,1,1,1,1)}$ is
forced to be a multiple of the regular representation alone.
\end{rem}

\section{Open problems}
\label{sec:open}

\begin{prob}[Extension to the species $\Kb_\alpha$]
\label{prob:K}
For $\Kb_\alpha \times \Kb_\beta$, the relevant group is
$G_\alpha \times G_\beta$, where $G_\alpha = \langle\sigma_{a_1}
\rangle \times \cdots \times \langle\sigma_{a_k}\rangle$ is a
product of cyclic groups but is \emph{not itself cyclic} once two
parts of $\alpha$ coincide. The group $G_\alpha \times G_\beta$
is still abelian, so Lemma~\ref{lem:orbit_invariant} continues to
hold and stabilizers remain orbit-invariant. However,
Lemma~\ref{lem:unique_subgroup} fails: $G_\alpha$ may have several
subgroups of the same order, so the M\"obius inversion of
Theorem~\ref{thm:mobius} must be carried out over the full
subgroup lattice of $G_\alpha$ (a product of divisor lattices)
rather than over the divisors of a single integer. Determine
whether this subgroup-lattice M\"obius inversion still yields a
closed-form formula for the coefficients $j^\lambda_{\alpha,\beta}$ in
$\Kb_\alpha \times \Kb_\beta = \sum_\lambda j^\lambda_{\alpha,\beta}\Kb_\lambda$.
\end{prob}

\begin{prob}[Non-commutative refinement]
\label{prob:ribbon}
The paper~\cite{BR2026} also expands $\Cb_\lambda$ in the ribbon
basis $\{r_\nu\}$ of non-commutative symmetric functions. Determine
whether the periodic patterns of Section~\ref{sec:realization}
carry a natural descent statistic explaining the resulting
expansion coefficients $\sigma_{\lambda,\nu}$ combinatorially, in
the same spirit as Theorem~\ref{thm:mobius} explains
$b^\lambda_{\alpha,\beta}$ algebraically.
\end{prob}

\section{Conclusion}
\label{sec:conclusion}
We realized $\Cb_\alpha$ geometrically as an infinite multi-row
periodic pattern shifted by a single integer
(Proposition~\ref{prop:realization}), and then obtained a
closed-form formula for the coefficient $b^\lambda_{\alpha,\beta}$
via M\"obius inversion over the divisors of
$\gcd(o(\sigma_\alpha),o(\sigma_\beta))$ (Theorem~\ref{thm:mobius}).
This is structurally simpler than the Garsia--Remmel shuffle rule for
$h_\lambda \star h_\mu$, since the cyclic-subgroup setting reduces to
counting orbits of an \emph{abelian} group action with a totally
ordered subgroup lattice.
Section~\ref{sec:open} collects natural next steps, including
extending this technique to the species $\Kb_\alpha$
(Open Problem~\ref{prob:K}).

\section*{Acknowledgements}
The authors would like to thank François Bergeron, Martin Rubey and Matthieu Josuat-Vergès for their valuable feedback and helpful comments. All numerical computations and consistency checks in this work were carried out and independently verified using SageMath~\cite{SageMath}. Additionally, the authors acknowledge the use of Anthropic's Claude as an AI assistant for exploration and proof verification and to polish the English grammar and stylistic flow of the manuscript. The authors rigorously reviewed all AI-generated and computational output and assume full responsibility for the final mathematical results and text.


\end{document}